\documentclass[a4paper,12pt]{elsarticle}
\usepackage{CJK}
\usepackage{multirow}
\usepackage{booktabs}
\usepackage{amssymb}
\usepackage{comment}
\usepackage{mathrsfs}        
\usepackage{amsthm}          
\usepackage{amsmath}         
\usepackage{epstopdf}
\biboptions{numbers,sort&compress}
\usepackage{hyperref}
\usepackage{xcolor}  
\definecolor{b}{rgb}{0,0,1}
\usepackage{caption2}                   
\theoremstyle{plain}
\newtheorem{theorem}{Theorem}

\newtheorem{corollary}{Corollary}
\newtheorem{lemma}{Lemma}

\theoremstyle{definition}
\newtheorem{definition}{Definition}
\newtheorem{example}{Example}
\newtheorem{remark}{Remark}

\usepackage{bbding}         
 \usepackage[misc]{ifsym}   
\journal{Arxiv}

\begin{document}

\begin{frontmatter}



\title{Lyapunov functions for nabla discrete \\ [3pt] fractional order systems}
\phantomsection
\addcontentsline{toc}{title}{Lyapunov functions for nabla discrete fractional order systems}


\author{Yiheng Wei}
\author{Yuquan Chen}
\author{Tianyu Liu}
\author{Yong Wang}
\address{Department of Automation, University of Science and Technology of China, Hefei, 230026, China}

\begin{abstract}
This paper focuses on the fractional difference of Lyapunov functions related to Riemann--Liouville, Caputo and Gr\"{u}nwald--Letnikov definitions. A new way of building Lyapunov functions is introduced and then five inequalities are derived for each definition. With the help of the developed inequalities, the sufficient conditions can be obtained to guarantee the asymptotic stability of the nabla discrete fractional order nonlinear systems. Finally, three illustrative examples are presented to demonstrate the validity and feasibility of the proposed theoretical results.
\end{abstract}


\end{frontmatter}



\section{Introduction}\label{Section 1}
Present days, fractional calculus finds more attention in the science and engineering aspects because of its significant characteristics, such as nonlocality, long memory, and self-similarity, etc \cite{Wei:2017FCAAb,Sun:2018CNSNS,Wei:2019CNSNS}. In consideration of the outstanding capability of fractional calculus, many scholars and engineers have devoted themselves to such a potential field. Thanks to their research efforts, a large volume of papers have been published, covering system modeling \cite{Cheng:2018SP}, automatic control \cite{Zhou:2017ISA}, adaptive filtering \cite{Chen:2017AMC} and so forth.

Stability as an important index of control systems receives wide attention, and various kinds of stability issues of fractional order systems have been studied \cite{Wei:2017FCAAa,Alagoz:2017ISA}. It is no doubt that Lyapunov direct method plays a central role in the investigation of the stability, and one needs to calculate fractional derivatives of the Lyapunov functions according to previous works \cite{Li:2009Auto}. However, the well-known Leibniz rule does not hold for fractional derivatives \cite{Tarasov:2013CNSNS}, which brings great difficulty to stability analyses. In \cite{Trigeassou:2011SP,Chen:2017NODY}, the authors took two steps to judge the stability of a fractional order system. First, derive the equivalent frequency distributed model and construct the Lyapunov function with the exact state instead of the pseudo state \cite{Trigeassou:2013CMA}. Second, take an integer order derivative to the Lyapunov function. However, their main results are established on the basis of an unproven expression, i.e., if the exact state is convergent, the pseudo one will converge too.

Fortunately, an effective method was proposed to estimate the upper bound for fractional derivative of Lyapunov functions, which did not need to calculate the infinite sum \cite{Aguila:2014CNSNS}. In fact, the prototype of this method can be found in Corollary 1 of \cite{Alikhanov:2012AMC}. This method has attracted great attention and been widely used. To extend the method from scalar case to vector case, a general inequality was presented in \cite{Duarte:2015CNSNS}. To further enrich the form of Lyapunov functions, much excellent work has been done \cite{Ding:2015CTA,Fernandez:2017CNSNS,Fernandez:2018CNSNS,Dai:2017NODY}. Notably, those results only focus on the Caputo fractional derivative except \cite{Dai:2017NODY}. Though the Riemann--Liouville case was considered, the obtained result was still far from complete. Afterwards, a modified result was developed in \cite{Liu:2016NODY}. Likewise, some pioneering work was done for the discrete case with either Caputo definition \cite{Baleanu:2017CNSNS} or Riemann--Liouville definition \cite{Wu:2017AMC}. However, to our best knowledge, there was a gap in the literature concerning the fractional difference of Lyapunov functions. The inequalities on the Gr\"{u}nwald--Letnikov fractional difference of Lyapunov function are still scarce. More approaches for designing the required Lyapunov functions are expected.

Motivated by the aforementioned discussions, this study concerns with the fractional difference of Lyapunov functions. The main contributions are concluded here.
\begin{enumerate}[i)]
    \item A new way to examine the asymptotic stability of nabla discrete fractional order systems is proposed;
    \item Five inequalities on the fractional difference of Lyapunov functions are systematically developed;
    \item The applicability of the proposed results for three different definitions is rigorously demonstrated.
\end{enumerate}

The layout of this paper is organized as follows. Section \ref{Section 2} reviews the required definitions and some critical facts. The establishment of several essential inequalities for widely used fractional difference definitions is reported in Section \ref{Section 3}. In Section \ref{Section 4}, to further evaluate the effectiveness and practicability of the developed results, numerical simulation is performed. Finally, Section \ref{Section 5} draws the conclusions.

\section{Preliminaries}\label{Section 2}

In this section, some definitions of discrete fractional calculus and fundamental knowledge are introduced.

\begin{definition}\label{Definition 1}
The $n$-th integer order backward difference of a function $f:{\mathbb N}_{1-n}\to\mathbb{R}$ is defined as \cite{Ostalczyk:2015Book}
\begin{equation}\label{Eq1}
{\textstyle{\nabla}^{ n } f\left( k \right) \triangleq \sum^{n}_{j=0}(-1)^j \big( {\begin{smallmatrix}
n \\
j
\end{smallmatrix}}\big)f\left( k-j \right),}
\end{equation}
where $ n\in {\mathbb Z}_{+} $, $ k\in {\mathbb Z}_{+}$, $ {\mathbb N}_{1-n} \triangleq  \left\{ { 1-n,2-n,3-n, \cdots } \right\} $, $ \big( {\begin{smallmatrix}
p \\
q
\end{smallmatrix}}\big) \triangleq \frac{\Gamma (p + 1) }{ \Gamma ( q+1 ) \Gamma ( p - q + 1 ) }$ and $ \Gamma ( x ) \triangleq \int_{\rm{0}}^{{\rm{ + }}\infty } {{{\rm{e}}^t}{t^{x - 1}}} {\rm d}t $.
\end{definition}

Under this definition, the Leibniz rule of $1$-th backward difference can be expressed as
\begin{equation}\label{Eq2}
{\textstyle\begin{array}{rl}
\nabla \left\{ {f\left( k \right)g\left( k \right)} \right\}
 =&\hspace{-6pt} \nabla f\left( k \right)g\left( k \right) + f\left( {k - 1} \right)\nabla g\left( k \right)\\
 = &\hspace{-6pt} \nabla f\left( k \right)g\left( {k - 1} \right) + f\left( k \right)\nabla g\left( k \right),
\end{array}}
\end{equation}
from which the formulas of summation by parts follow
\begin{equation}\label{Eq3}
{\textstyle\sum\nolimits_{j = a}^k {f\left( {j - 1} \right)\nabla g\left( j \right)}  = f\left( j \right)g\left( j \right)|^k_{a-1}  - \sum\nolimits_{j = a}^k {\nabla f\left( j \right)g\left( j \right).} }
\end{equation}
\begin{equation}\label{Eq4}
{\textstyle\sum\nolimits_{j = a}^k {f\left( {j} \right)\nabla g\left( j \right)}  = f\left( j \right)g\left( j \right)|^k_{a-1}  - \sum\nolimits_{j = a}^k {\nabla f\left( j \right)g\left( j-1 \right)} ,}
\end{equation}

\begin{definition}\label{Definition 2}
The $\alpha$-th Gr\"{u}nwald--Letnikov difference/sum of a function $ f:{\mathbb N}_a \to {\mathbb R}$ is defined by \cite{Ostalczyk:2015Book}
\begin{equation}\label{Eq5}
{\textstyle{}_{a}^{\rm G}{\nabla}_{k}^{\alpha} f\left( k \right) \triangleq  \sum\nolimits_{j = 0}^{k-a} {{{\left( { - 1} \right)}^j}\big( {\begin{smallmatrix}
\alpha \\
j
\end{smallmatrix}} \big)f\left( {k - j} \right)},}
\end{equation}
where $\alpha<0$, $a\in\mathbb{R}$ and $ k\in {\mathbb N}_{a} $.
\end{definition}

For convenience, let $ t^{\overline {r} } \triangleq \frac{{\Gamma \left( {t + r} \right)}}{{\Gamma \left( {t} \right)}} $, $t\in\mathbb{N}$, $r\in\mathbb{R}$. Then the formula in (\ref{Eq5}) can be equivalently expressed as
\begin{equation}\label{Eq6}
{\textstyle\begin{array}{rl}
{}_{a}^{\rm G}{\nabla}_{k}^{\alpha} f\left( k \right) =&\hspace{-6pt}  \sum\nolimits_{j = 0}^{k-a} {\frac{{{{\left( {j + 1} \right)}^{\overline { - \alpha  - 1} }}}}{{\Gamma \left( { - \alpha } \right)}}f\left( {k - j}
\right)}\\
=&\hspace{-6pt}  \sum\nolimits_{j = a}^{k} {\frac{{{{\left( {k-j + 1} \right)}^{\overline { - \alpha  - 1} }}}}{{\Gamma \left( { - \alpha } \right)}}f\left( {j}\right)}.
\end{array}}
\end{equation}

\begin{definition}\label{Definition 3}
Assume $f : \mathbb{N}_{a-n} \to \mathbb{R}$, $\alpha \in (n-1,n)$ and $n\in\mathbb{Z}_+$. Then the nabla Riemann--Liouville and Caputo fractional differences (if they exist) is defined as \cite{Ostalczyk:2015Book}
\begin{equation}\label{Eq7}
{}_a^{\rm R}{\nabla}_k ^\alpha f\left( k \right) \triangleq {\nabla ^n}{}_a^{\rm G}{\nabla}_k ^{ \alpha-n }f\left(k \right),
\end{equation}
\begin{equation}\label{Eq8}
{}_a^{\rm C}{\nabla }_k ^\alpha f\left( k \right) \triangleq {}_a^{\rm G}{\nabla}_k ^{ \alpha-n }{\nabla ^n}f\left( k \right),
\end{equation}
respectively.
\end{definition}

With Definition \ref{Definition 2} and Definition \ref{Definition 3}, one has the initial value at $k=a$
\begin{equation}\label{Eq9}
{\textstyle {}_a^{\rm G}{\nabla}_k ^\alpha f\left( k \right)\big|_{k=a}=f\left( a \right),}
\end{equation}
\begin{equation}\label{Eq10}
{\textstyle {}_a^{\rm R}{\nabla}_k ^\alpha f\left( k \right)\big|_{k=a}=0,}
\end{equation}
\begin{equation}\label{Eq11}
{\textstyle {}_a^{\rm C}{\nabla}_k ^\alpha f\left( k \right)\big|_{k=a}={\nabla ^n}f\left( a \right).}
\end{equation}

The reference \cite{Ostalczyk:2015Book} also shows that
\begin{equation}\label{Eq12}
{\textstyle {}_a^{\rm{R}}\nabla _k^\alpha f\left( k \right) = {}_a^{\rm{C}}\nabla _k^\alpha f\left( k \right) + \sum\nolimits_{j = 0}^{n - 1} {\frac{{{{\left( {k - a + 1} \right)}^{\overline {j - \alpha } }}}}{{\Gamma \left( {j - \alpha  + 1} \right)}}} {\nabla ^j}f\left( {a - 1} \right)}.
\end{equation}

Before ending up the preliminaries, the Young inequality is given here, which plays an important role in the work.
\begin{lemma}\label{Lemma 1}
If $p,q>1$ and $\frac{1}{p}+\frac{1}{q}=1$, then \cite{Young:1936AM}
\begin{equation}\label{Eq13}
{\textstyle ab\le\frac{1}{p}a^p+\frac{1}{q}b^q},
\end{equation}
holds for any $a,b\ge0$. The equality holds if and only if $a^p=b^q$.
\end{lemma}
\section{Main Results}\label{Section 3}
This section presents three theorems of Lyapunov inequality for the fractional difference, which play an essential role in the stability analysis of discrete fractional order systems.
\subsection{Lyapunov inequalities}
\begin{theorem}\label{Theorem 1} The following inequalities hold
\begin{equation}\label{Eq14}
{\textstyle {}_a^{\rm C}\nabla _k^\alpha {x^{2m}}\left( k \right) \le 2x^m\left( k \right){}_a^{\rm C}\nabla _k^\alpha x^m\left( k \right)},
\end{equation}
\begin{equation}\label{Eq15}
{\textstyle {}_a^{\rm C}\nabla _k^\alpha {x^{\frac{{2m}}{n}}}\left( k \right) \le \frac{{2m}}{{2m - n}}x\left( k \right){}_a^{\rm C}\nabla _k^\alpha {x^{\frac{{2m}}{n} - 1}}\left( k \right)},
\end{equation}
\begin{equation}\label{Eq16}
{\textstyle {}_a^{\rm C}\nabla _k^\alpha {x^{\frac{{2m}}{n}}}\left( k \right) \le \frac{{2m}}{n}{x^{\frac{{2m}}{n} - 1}}\left( k \right){}_a^{\rm C}\nabla _k^\alpha x\left( k \right)},
\end{equation}
\begin{equation}\label{Eq17}
{\textstyle {}_a^{\rm C}\nabla _k^\alpha {x^{{2^m}}}\left( k \right) \le {2^m}{x^{{2^m} - 1}}\left( k\right){}_a^{\rm C}\nabla _k^\alpha x\left( k \right)},
\end{equation}
\begin{equation}\label{Eq18}
{\textstyle {}_a^{\rm C}\nabla _k^\alpha {y^{\rm{T}}}\left( k \right)Py\left( k \right) \le 2{y^{\rm{T}}}\left(k \right)P {}_a^{\rm C}\nabla _k^\alpha y\left( k \right)},
\end{equation}
for any $0<\alpha<1$, $m\in\mathbb{Z}_+$, $n\in\mathbb{Z}_+$, $2m\ge n$, $x(k)\in\mathbb{R}$, $y(k)\in\mathbb{R}^\kappa$, $\kappa\in\mathbb{Z}_+$, $ k\in {\mathbb N}_{a} $, $a\in\mathbb{R}$ and positive definite matrix $P\in\mathbb{R}^{\kappa\times \kappa}$.
\end{theorem}
\begin{proof}
i) When $k\in\mathbb{N}_{a+1}$, based on Definition \ref{Definition 3} and formula (\ref{Eq6}), it follows
\begin{eqnarray}\label{Eq19}
{\textstyle \begin{array}{l}
{}_a^{\rm C}\nabla _k^\alpha {x^{2m}}\left( k \right) - 2x^m\left(k \right){}_a^{\rm C}\nabla _k^\alpha x^m\left( k \right)\\
= {}_a^{\rm G }\nabla _k^{\alpha  - 1}\nabla {x^{2m}}\left( k \right) - 2x^m\left( k \right){}_a^{\rm G}\nabla _k^{\alpha  - 1}\nabla x^m\left( k \right)\\
= \sum\nolimits_{j = a}^{k} {{\frac{{{{\left( {k - j + 1} \right)}^{\overline { - \alpha} }}}}{{\Gamma \left( { - \alpha +1} \right)}}} \left[ {x^{2m}}\left( j \right) - {x^{2m}}\left( j-1 \right) \right] } \\
\hspace{10pt} - 2x^m\left( k \right) \sum\nolimits_{j = a}^{k} {{\frac{{{{\left( {k - j + 1} \right)}^{\overline { - \alpha} }}}}{{\Gamma \left( { - \alpha +1} \right)}}} \left[ {x}^m\left( j \right) - {x}^m\left( j-1 \right) \right] } \\
=  \sum\nolimits_{j = a}^{k} {{\frac{{{{\left( {k - j + 1} \right)}^{\overline { - \alpha} }}}}{{\Gamma \left( { - \alpha +1} \right)}}}} \{ {x^{2m}}\left( j \right) - {x^{2m}}\left( j-1 \right) - 2x^m\left( k \right) \left[ {x}\left( j \right) - {x}^m\left( j-1 \right) \right] \} \\
=  \sum\nolimits_{j= a}^{k} {{\frac{{{{\left( {k - j + 1} \right)}^{\overline { - \alpha} }}}}{{\Gamma \left( { - \alpha +1} \right)}}}} \{ { \left[ x^m\left( k \right) - x^m\left( j \right) \right]}^2 - { \left[ x^m\left( k \right) - x^m\left( j-1 \right) \right]}^2 \}\\
=  \sum\nolimits_{j= a}^{k} {{\frac{{{{\left( {k - j + 1} \right)}^{\overline { - \alpha} }}}}{{\Gamma \left( { - \alpha +1} \right)}}}} \nabla { \left[ x^m\left( k \right) - x^m\left( j \right) \right]}^2 .
\end{array}}
\end{eqnarray}

Setting $f\left( {j-1} \right) = \frac{{{{\left( {k - j + 1} \right)}^{\overline { - \alpha } }}}}{{\Gamma \left( { - \alpha  + 1} \right)}},~g\left( j \right) = \left[x^m\left( k \right) - x^m\left( j \right)\right]^2$ and the formula of summation by parts (\ref{Eq3}), one can obtain
\begin{equation}\label{Eq20}
{\textstyle \begin{array}{l}
{}_a^{\rm C}\nabla _k^\alpha {x^{2m}}\left( k \right) - 2x^m\left( k \right){}_a^{\rm C}\nabla _k^\alpha x^m\left( k \right) \\
=\sum\nolimits_{j = a}^k {f\left( {j-1} \right)\nabla g\left( j \right)}\\
= f\left( j \right)g\left( j \right)|^k_{a-1}  - \sum\nolimits_{j = a}^k {\nabla f\left( j \right)g\left( j \right)}\\
= \frac{{{0^{\overline { - \alpha } }}}}{{\Gamma \left( { - \alpha  + 1} \right)}}{g}\left( k \right) - \frac{{{{\left( {k - a + 1} \right)}^{\overline { - \alpha } }}}}{{\Gamma \left( { - \alpha  + 1} \right)}}{g}\left( {a - 1} \right) +\sum\nolimits_{j = a}^k { \frac{{{{\left( {k - j+1} \right)}^{\overline { - \alpha-1 } }}}}{{\Gamma \left( { - \alpha} \right)}}{g}\left( j \right)} ,
\end{array}}
\end{equation}
where $\nabla \frac{{{{\left( {k - j } \right)}^{\overline { - \alpha } }}}}{{\Gamma \left( { - \alpha  + 1} \right)}} =  - \frac{{{{\left( {k - j + 1} \right)}^{\overline { - \alpha  - 1} }}}}{{\Gamma \left( { - \alpha } \right)}}$ is adopted here.

Due to the fact that $\frac{{{0^{\overline { - \alpha } }}}}{{\Gamma \left( { - \alpha  + 1} \right)}} = 0$ and $\frac{{{1^{\overline { - \alpha } }}}}{{\Gamma \left( { - \alpha  + 1} \right)}} = 1$ hold for any $\alpha\in(0,1)$, the mentioned formula can be updated as
\begin{equation}\label{Eq21}
{\textstyle \begin{array}{l}
{}_a^{\rm C}\nabla _k^\alpha {x^{2m}}\left( k \right) - 2x^m\left( k \right){}_a^{\rm C}\nabla _k^\alpha x^m\left( k \right) \\
= g\left( k \right) - \frac{{{{\left( {k - a + 1} \right)}^{\overline { - \alpha } }}}}{{\Gamma \left( { - \alpha  + 1} \right)}}{g}\left( {a - 1} \right) +\sum\nolimits_{j = a}^{k-1} { \frac{{{{\left( {k - j+1} \right)}^{\overline { - \alpha-1 } }}}}{{\Gamma \left( { - \alpha} \right)}}{g}\left( j \right)} .
\end{array}}
\end{equation}

Recalling the property of Gamma function, one obtains
\begin{equation}\label{Eq22}
{\textstyle {g}\left( k \right) = 0,k \in \mathbb{N}_{a},}
\end{equation}
\begin{equation}\label{Eq23}
{\textstyle {g}\left( j \right) \ge 0,j =a-1,a,\cdots, k-1,}
\end{equation}
\begin{equation}\label{Eq24}
{\textstyle \frac{{{{\left( {k - a + 1} \right)}^{\overline { - \alpha } }}}}{{\Gamma \left( { - \alpha  + 1} \right)}} \ge 0,k \in \mathbb{N}_{a+1},}
\end{equation}
\begin{equation}\label{Eq25}
{\textstyle \frac{{{{\left( {k - j + 1} \right)}^{\overline { - \alpha  - 1} }}}}{{\Gamma \left( { - \alpha } \right)}}\le 0,j = a,a + 1, \cdots ,k-1.}
\end{equation}
Combining equations (\ref{Eq21})-(\ref{Eq25}), one has
\begin{equation}\label{Eq26}
{\textstyle \begin{array}{l}
{}_a^{\rm C}\nabla _k^\alpha {x^{2m}}\left( k \right) - 2x^m\left( k \right){}_a^{\rm C}\nabla _k^\alpha x^m\left( k \right)\\
= - \frac{{{{\left( {k - a + 1} \right)}^{\overline { - \alpha } }}}}{{\Gamma \left( { - \alpha  + 1} \right)}}{g}\left( {a - 1} \right) +\sum\nolimits_{j = a}^{k-1} { \frac{{{{\left( {k - j+1} \right)}^{\overline { - \alpha-1 } }}}}{{\Gamma \left( { - \alpha} \right)}}{g}\left( j \right)} \\
\le 0,
\end{array}}
\end{equation}
which confirms formula (\ref{Eq14}) with $k\in \mathbb{N}_{a+1}$.

When $k=a$, by applying(\ref{Eq11}), the desirable formula (\ref{Eq14}) becomes
\begin{equation}\label{Eq27}
{\textstyle
{\nabla ^1}{x^{2m}}\left( a \right) \le 2{x^m}\left( a \right){\nabla ^1}{x^m}\left( a \right).}
\end{equation}
Expanding the formula yields
\begin{equation}\label{Eq28}
{\textstyle
\begin{array}{l}
{\nabla ^1}{x^{2m}}\left( a \right)-2{x^m}\left( a \right){\nabla ^1}{x^m}\left( a \right)\\
 = \left[ {{x^{2m}}\left( a \right) - {x^{2m}}\left( {a - 1} \right)} \right]-2{x^m}\left( a \right)\left[ {{x^m}\left( a \right) - {x^m}\left( {a - 1} \right)} \right]\\
 = -{\left[ {{x^m}\left( a \right) - {x^m}\left( {a - 1} \right)} \right]^2}\\
 \le 0,
\end{array}}
\end{equation}
which completes the proof of formula (\ref{Eq14}).

ii) Analogously, it follows from $k\in\mathbb{N}_{a+1}$ that
\begin{equation}\label{Eq29}
{\textstyle
\begin{array}{l}
{}_a^{\rm C}\nabla _k^\alpha {x^{\frac{{2m}}{n}}}\left( k \right) - \frac{{2m}}{{2m - n}}x\left( k \right){}_a^{\rm C}\nabla _k^\alpha {x^{\frac{{2m}}{n} - 1}}\left( k \right)\\
 = {}_a^{ \rm G}\nabla _k^{\alpha  - 1}\nabla {x^{\frac{{2m}}{n}}}\left( k \right) - \frac{{2m}}{{2m - n}}x\left( k \right){}_a^{\rm G}\nabla _k^{\alpha  - 1}\nabla {x^{\frac{{2m}}{n} - 1}}\left( k \right)\\
 = \sum\nolimits_{j = a}^k {\frac{{{{\left( {k - j + 1} \right)}^{\overline { - \alpha } }}}}{{\Gamma \left( { - \alpha  + 1} \right)}}\frac{{2m}}{n}{x^{\frac{{2m}}{n} - 1}}\left( j \right)\nabla x\left( j \right)} \\
\hspace{12pt} - x\left( k \right)\sum\nolimits_{j = a}^k {\frac{{{{\left( {k - j + 1} \right)}^{\overline { - \alpha } }}}}{{\Gamma \left( { - \alpha  + 1} \right)}}\frac{{2m}}{n}{x^{\frac{{2m}}{n} - 2}}\left( j \right)\nabla x\left( j \right)} \\
 = \sum\nolimits_{j = a}^k {f\left( j-1 \right)\nabla g\left( j \right)},
\end{array}}
\end{equation}
where $f\left( j \right) = \frac{{{{\left( {k - j} \right)}^{\overline { - \alpha } }}}}{{\Gamma \left( { - \alpha  + 1} \right)}}$ and $\nabla g\left( j \right) = \frac{{2m}}{n}\left[ {x\left( j \right) - x\left( k \right)} \right]{x^{\frac{{2m}}{n} - 2}}\left( j \right)\nabla x\left( j \right)$. Notably, to facilitate the subsequent proof, two extra terms which are free of $j$ are introduced to construct $g(j)$, namely
\begin{equation}\label{Eq30}
{\textstyle
g\left( {j} \right) = {x^{\frac{{2m}}{n}}}\left( j \right) - \frac{{2m}}{{2m - n}}{x^{\frac{{2m}}{n} - 1}}\left( j \right)x\left( k \right) - {x^{\frac{{2m}}{n}}}\left( k \right) + \frac{{2m}}{{2m - n}}{x^{\frac{{2m}}{n}}}\left( k \right).}
\end{equation}

By applying the formula of summation by parts in (\ref{Eq3}), one has
\begin{equation}\label{Eq31}
{\textstyle \begin{array}{l}
{}_a^{\rm C}\nabla _k^\alpha {x^{\frac{{2m}}{n}}}\left( k \right) - \frac{{2m}}{{2m - n}}x\left( k \right){}_a^{\rm C}\nabla _k^\alpha {x^{\frac{{2m}}{n} - 1}}\left( k \right)\\
= g\left( k \right) - \frac{{{{\left( {k - a + 1} \right)}^{\overline { - \alpha } }}}}{{\Gamma \left( { - \alpha  + 1} \right)}}{g}\left( {a - 1} \right) +\sum\nolimits_{j = a}^{k-1} { \frac{{{{\left( {k - j+1} \right)}^{\overline { - \alpha-1 } }}}}{{\Gamma \left( { - \alpha} \right)}}{g}\left( j \right)} .
\end{array}}
\end{equation}
Interestingly, this formula is identical to (\ref{Eq21}). Besides, the defined $f(j)$ and $g(j)$ meet the conditions in (\ref{Eq22}), (\ref{Eq24}) and (\ref{Eq25}). In other words, if formula (\ref{Eq23}) still holds for $g(j)$ in (\ref{Eq30}), the desirable result
\begin{equation}\label{Eq32}
{\textstyle \begin{array}{l}
{}_a^{\rm C}\nabla _k^\alpha {x^{\frac{{2m}}{n}}}\left( k \right) - \frac{{2m}}{{2m - n}}x\left( k \right){}_a^{\rm C}\nabla _k^\alpha {x^{\frac{{2m}}{n} - 1}}\left( k \right)\\
= - \frac{{{{\left( {k - a + 1} \right)}^{\overline { - \alpha } }}}}{{\Gamma \left( { - \alpha  + 1} \right)}}{g}\left( {a - 1} \right) +\sum\nolimits_{j = a}^{k-1} { \frac{{{{\left( {k - j+1} \right)}^{\overline { - \alpha-1 } }}}}{{\Gamma \left( { - \alpha} \right)}}{g}\left( j \right)} \\
\le 0,
\end{array}}
\end{equation}
can be proven to be true.

Defining $a = \big| {{x^{\frac{{2m}}{n} - 1}}\left( j \right)} \big|$, $b = \left| {x\left( k \right)} \right|$, $p = \frac{{2m}}{{2m - n}}$, $q = \frac{{2m}}{n}$ and using Lemma \ref{Lemma 1}, then one has
\begin{equation}\label{Eq33}
{\textstyle \begin{array}{rl}
{x^{\frac{{2m}}{n} - 1}}\left( j \right)x\left( k \right) \le&\hspace{-6pt} \big| {{x^{\frac{{2m}}{n} - 1}}\left( j \right)} \big|\left| {x\left( k \right)} \right|\\
 \le&\hspace{-6pt} \frac{{2m - n}}{{2m}}{\big| {{x^{\frac{{2m}}{n} - 1}}\left( j \right)} \big|^{\frac{{2m}}{{2m - n}}}} + \frac{n}{{2m}}{\left| {x\left( k \right)} \right|^{\frac{{2m}}{n}}}\\
 =&\hspace{-6pt} \frac{{2m - n}}{{2m}}{x^{\frac{{2m}}{n}}}\left( j \right) + \frac{n}{{2m}}{x^{\frac{{2m}}{n}}}\left( k \right).
\end{array}}
\end{equation}
Furthermore, $g(j)$ can be reduced to
\begin{equation}\label{Eq34}
{\textstyle \begin{array}{rl}
g\left({j} \right) \ge&\hspace{-6pt} {x^{\frac{{2m}}{n}}}\left( j \right) - \frac{{2m}}{{2m - n}}\big[ {\frac{{2m - n}}{{2m}}{x^{\frac{{2m}}{n}}}\left( j \right) + \frac{n}{{2m}}{x^{\frac{{2m}}{n}}}\left( k \right)} \big]\\
 &\hspace{-6pt}- {x^{\frac{{2m}}{n}}}\left( k \right) + \frac{{2m}}{{2m - n}}{x^{\frac{{2m}}{n}}}\left( k \right)\\
 =&\hspace{-6pt} {x^{\frac{{2m}}{n}}}\left( j \right) - {x^{\frac{{2m}}{n}}}\left( j \right) - \frac{n}{{2m - n}}{x^{\frac{{2m}}{n}}}\left( k \right)\\
 &\hspace{-6pt}- {x^{\frac{{2m}}{n}}}\left( k \right) + \frac{{2m}}{{2m - n}}{x^{\frac{{2m}}{n}}}\left( k \right)\\
 =&\hspace{-6pt} 0,
\end{array}}
\end{equation}
from which one can conclude that (\ref{Eq15}) holds for $k\in\mathbb{N}_{a+1}$.

Assuming $k=a$ for (\ref{Eq15}), one has
\begin{equation}\label{Eq35}
{\textstyle {\nabla ^1}{x^{\frac{{2m}}{n}}}\left( a \right) \le \frac{{2m}}{{2m - n}}x\left( a \right){\nabla ^1}{x^{\frac{{2m}}{n} - 1}}\left( a \right).}
\end{equation}
Then,  basic mathematical operation leads to
\begin{equation}\label{Eq36}
{\textstyle \begin{array}{l}
{\nabla ^1}{x^{\frac{{2m}}{n}}}\left( a \right) - \frac{{2m}}{{2m - n}}x\left( a \right){\nabla ^1}{x^{\frac{{2m}}{n} - 1}}\left( a \right)\\
 = \frac{{2m}}{{2m - n}}x\left( a \right){x^{\frac{{2m}}{n} - 1}}\left( {a - 1} \right) - \frac{n}{{2m - n}}{x^{\frac{{2m}}{n}}}\left( a \right) - {x^{\frac{{2m}}{n}}}\left( {a - 1} \right)\\
 \le \frac{{2m}}{{2m - n}}\big[ {\frac{n}{{2m}}{x^{\frac{{2m}}{n}}}\left( a \right) + \frac{{2m - n}}{{2m}}{x^{\frac{{2m}}{n}}}\left( {a - 1} \right)} \big]\\
 \hspace{12pt}- \frac{n}{{2m - n}}{x^{\frac{{2m}}{n}}}\left( a \right) - {x^{\frac{{2m}}{n}}}\left( {a - 1} \right)\\
 \le 0.
\end{array}}
\end{equation}
The proof of (\ref{Eq15}) is thus completed.

iii) Likewise, the following formula with $k\in\mathbb{N}_{a+1}$ can be derived
\begin{equation}\label{Eq37}
{\textstyle \begin{array}{l}
{}_a^{\rm C}\nabla _k^\alpha {x^{\frac{{2m}}{n}}}\left( k \right) - \frac{{2m}}{n}{x^{\frac{{2m}}{n} - 1}}\left( k \right){}_a^{\rm C}\nabla _k^\alpha x\left( k \right)\\
 = {}_a^{\rm G }\nabla _k^{\alpha  - 1}\nabla {x^{\frac{{2m}}{n}}}\left( k \right) - \frac{{2m}}{n}{x^{\frac{{2m}}{n} - 1}}\left( k \right){}_a^{\rm G }\nabla _k^{\alpha  - 1}\nabla x\left( k \right)\\
 = \sum\nolimits_{j = a}^k {\frac{{{{\left( {k - j + 1} \right)}^{\overline { - \alpha } }}}}{{\Gamma \left( { - \alpha  + 1} \right)}}\frac{{2m}}{n}{x^{\frac{{2m}}{n} - 1}}\left( j \right)\nabla x\left( j \right)} \\
 \hspace{12pt}- \frac{{2m}}{n}{x^{\frac{{2m}}{n} - 1}}\left( k \right)\sum\nolimits_{j = a}^k {\frac{{{{\left( {k - j + 1} \right)}^{\overline { - \alpha } }}}}{{\Gamma \left( { - \alpha  + 1} \right)}}\nabla x\left( j \right)} \\
 = \sum\nolimits_{j = a}^k {f\left( j-1 \right)\nabla g\left( j \right)},
\end{array}}
\end{equation}
where $f\left( j \right) = \frac{{{{\left( {k - j} \right)}^{\overline { - \alpha } }}}}{{\Gamma \left( { - \alpha  + 1} \right)}}$ and $\nabla g\left( j \right) = \frac{{2m}}{n}\big[ {{x^{\frac{{2m}}{n} - 1}}\left( j \right) - {x^{\frac{{2m}}{n} - 1}}\left( k \right)} \big]\nabla x\left( j \right)$. Also, to counteract the effects of the cross term, two terms free of $j$ are introduced, to be exact,
\begin{equation}\label{Eq38}
{\textstyle g\left( j \right) = {x^{\frac{{2m}}{n}}}\left( j \right) - \frac{{2m}}{n}{x^{\frac{{2m}}{n} - 1}}\left( k \right)x\left( j \right) - {x^{\frac{{2m}}{n}}}\left( k \right) + \frac{{2m}}{n}{x^{\frac{{2m}}{n}}}\left( k \right).}
\end{equation}

According to (\ref{Eq3}), the formula in (\ref{Eq37}) can be updated as
\begin{equation}\label{Eq39}
{\textstyle\begin{array}{l}
{}_a^{\rm C}\nabla _k^\alpha {x^{\frac{{2m}}{n}}}\left( k \right) - \frac{{2m}}{n}{x^{\frac{{2m}}{n} - 1}}\left( k \right){}_a^{\rm C}\nabla _k^\alpha x\left( k \right)\\
 = g\left( k \right) - \frac{{{{\left( {k - a + 2} \right)}^{\overline { - \alpha } }}}}{{\Gamma \left( { - \alpha  + 1} \right)}}g\left( {a - 1} \right) + \sum\nolimits_{j = a}^k {\frac{{{{\left( {k - j + 2} \right)}^{\overline { - \alpha  - 1} }}}}{{\Gamma \left( { - \alpha } \right)}}g\left( {j - 1} \right)}.
\end{array}}
\end{equation}
Similarly, if the variable $g(j)$ is nonnegative for $j=a-1,a,\cdots,k$, the following expression holds
\begin{equation}\label{Eq40}
{\textstyle \begin{array}{l}
{}_a^{\rm C}\nabla _k^\alpha {x^{\frac{{2m}}{n}}}\left( k \right) - \frac{{2m}}{n}{x^{\frac{{2m}}{n} - 1}}\left( k \right){}_a^{\rm C}\nabla _k^\alpha x\left( k \right)\\
= - \frac{{{{\left( {k - a + 1} \right)}^{\overline { - \alpha } }}}}{{\Gamma \left( { - \alpha  + 1} \right)}}{g}\left( {a - 1} \right) +\sum\nolimits_{j = a}^{k-1} { \frac{{{{\left( {k - j+1} \right)}^{\overline { - \alpha-1 } }}}}{{\Gamma \left( { - \alpha} \right)}}{g}\left( j \right)} \\
\le 0,
\end{array}}
\end{equation}

Setting $a = \big| {{x^{\frac{{2m}}{n} - 1}}\left( k \right)} \big|$, $b = \left| {x\left( j \right)} \right|$, $p = \frac{{2m}}{{2m - n}}$, $q = \frac{{2m}}{n}$ and using Lemma \ref{Lemma 1}, one obtains
\begin{equation}\label{Eq41}
{\textstyle \begin{array}{rl}
{x^{\frac{{2m}}{n} - 1}}\left( k \right)x\left( j \right) \le&\hspace{-6pt} \big| {{x^{\frac{{2m}}{n} - 1}}\left( k \right)} \big|\left| {x\left( j \right)} \right|\\
\le&\hspace{-6pt} \frac{{2m - n}}{{2m}}{\big| {{x^{\frac{{2m}}{n} - 1}}\left( k \right)} \big|^{\frac{{2m}}{{2m - n}}}} + \frac{n}{{2m}}{\left| {x\left( j \right)} \right|^{\frac{{2m}}{n}}}\\
 = &\hspace{-6pt} \frac{{2m - n}}{{2m}}{x^{\frac{{2m}}{n}}}\left( k \right) + \frac{n}{{2m}}{x^{\frac{{2m}}{n}}}\left( j \right),
\end{array}}
\end{equation}
and therefore $g(j)$ can be rewritten as
\begin{equation}\label{Eq42}
{\textstyle \begin{array}{rl}
g\left( j \right) \ge&\hspace{-6pt} {x^{\frac{{2m}}{n}}}\left( j \right) - \frac{{2m}}{n}\big[ {\frac{{2m - n}}{{2m}}{x^{\frac{{2m}}{n}}}\left( k \right) + \frac{n}{{2m}}{x^{\frac{{2m}}{n}}}\left( j \right)} \big]\\
 &\hspace{-6pt}- {x^{\frac{{2m}}{n}}}\left( k \right) + \frac{{2m}}{n}{x^{\frac{{2m}}{n}}}\left( k \right)\\
 =&\hspace{-6pt} {x^{\frac{{2m}}{n}}}\left( j \right) - \frac{{2m - n}}{n}{x^{\frac{{2m}}{n}}}\left( k \right) - {x^{\frac{{2m}}{n}}}\left( j \right)\\
 &\hspace{-6pt}- {x^{\frac{{2m}}{n}}}\left( k \right) + \frac{{2m}}{n}{x^{\frac{{2m}}{n}}}\left( k \right)\\
 =&\hspace{-6pt} 0.
\end{array}}
\end{equation}

When $k=a$, (\ref{Eq16}) reduces to
\begin{equation}\label{Eq43}
{\textstyle {\nabla ^1}{x^{\frac{{2m}}{n}}}\left( a \right) \le \frac{{2m}}{n}{x^{\frac{{2m}}{n} - 1}}\left( a \right){\nabla ^1}x\left( a \right).}
\end{equation}
After basic operation, it follows
\begin{equation}\label{Eq44}
{\textstyle \begin{array}{l}
{\nabla ^1}{x^{\frac{{2m}}{n}}}\left( a \right) - \frac{{2m}}{n}{x^{\frac{{2m}}{n} - 1}}\left( a \right){\nabla ^1}x\left( a \right)\\
 = \frac{{2m}}{n}{x^{\frac{{2m}}{n} - 1}}\left( a \right)x\left( {a - 1} \right) - \frac{{2m - n}}{n}{x^{\frac{{2m}}{n}}}\left( a \right) - {x^{\frac{{2m}}{n}}}\left( {a - 1} \right)\\
 \le \frac{{2m}}{n}\big[ {\frac{{2m - n}}{{2m}}{x^{\frac{{2m}}{n}}}\left( a \right) + \frac{n}{{2m}}{x^{\frac{{2m}}{n}}}\left( {a - 1} \right)} \big]\\
 \hspace{12pt}- \frac{{2m - n}}{n}{x^{\frac{{2m}}{n}}}\left( a \right) - {x^{\frac{{2m}}{n}}}\left( {a - 1} \right)\\
 \le 0
\end{array}}
\end{equation}
This completes the proof of (\ref{Eq16}).

iv) Setting $m=1$, the result in (\ref{Eq14}) can be represented as
\begin{equation}\label{Eq45}
{\textstyle \begin{array}{l}
{}_a^{\rm C}\nabla _k^\alpha {x^2}\left( k \right) - 2x\left( k \right){}_a^{\rm C}\nabla _k^\alpha x\left( k \right) \le 0.
\end{array}}
\end{equation}

Repeating the above-mentioned result in (\ref{Eq45}), the following formula
\begin{equation}\label{Eq46}
{\textstyle \begin{array}{rl}
{}_a^{\rm C}\nabla _k^\alpha {x^{{2^m}}}\left( k \right) \le&\hspace{-6pt} 2{x^{{2^{m - 1}}}}\left( k \right){}_a^{\rm C}\nabla _k^\alpha {x^{{2^{m - 1}}}}\left( k \right)\\
 \le&\hspace{-6pt} {2^2}{x^{{2^{m - 1}} + {2^{m - 2}}}}\left( k \right){}_a^{\rm C}\nabla _k^\alpha {x^{{2^{m - 2}}}}\left( k \right)\\
 \vdots \hspace{3pt}\\
 \le&\hspace{-6pt} {2^m}{x^{{2^m} - 1}}\left( k \right){}_a^{\rm C}\nabla _k^\alpha x\left( k \right),
\end{array} }
\end{equation}
can be derived for any $m\in\mathbb{Z}_+$. That is to say, (\ref{Eq17}) holds.

v) Due to the fact that for any positive definite matrix $P$ there exists a nonsingular matrix $M$ such that $P = {M^{\rm T}}M$, one will rearrange the target variable of (\ref{Eq18}) as
\begin{equation}\label{Eq47}
{\textstyle z\left( k \right) = My\left( k \right), }
\end{equation}
with $z\left( k \right)=\left[ {\begin{array}{*{20}{c}}
{{z_1}\left( k \right)}&{{z_2}\left( k \right)}& \cdots &{{z_\kappa }\left( k \right)}
\end{array}} \right]^{\rm{T}}$.

Then the multivariable case of (\ref{Eq45})
\begin{equation}\label{Eq48}
{\textstyle \begin{array}{l}
\vspace{2pt}
{}_a^{\rm C}\nabla _k^\alpha {y^{\rm{T}}}\left( k \right)Py\left( k \right) - 2{y^{\rm{T}}}\left( k \right)P {}_a^{\rm C}\nabla _k^\alpha y\left( k \right)\\
= {}_a^{\rm C}\nabla_k^\alpha z^{\rm{T}}\left( k \right) z\left( k \right)- 2{z}^{\rm{T}}\left( k \right) {}_a^{\rm C}\nabla _k^\alpha {z}\left( k \right) \\
= \sum\nolimits_{i = 1}^\kappa  {\left[ {{}_a^{\rm{C}}\nabla _k^\alpha z_i^2\left( k \right) - 2{z_i}\left( k \right){}_a^{\rm{C}}\nabla _k^\alpha {z_i}\left( k \right)} \right]}  \\
 \le 0
\end{array} }
\end{equation}
follows instantly.
\end{proof}

Configuring some parameters, the following corollary can be obtained from Theorem \ref{Theorem 1} immediately.

\begin{corollary}\label{Corollary 1} The following inequalities hold
\begin{equation}\label{Eq49}
{\textstyle {}_a^{\rm C}\nabla _k^\alpha {x^{2m}}\left( k \right) \le \frac{{2m}}{{2m - 1}}x\left( k \right){}_a^{\rm C}\nabla _k^\alpha {x^{2m - 1}}\left( k \right)},
\end{equation}
\begin{equation}\label{Eq50}
{\textstyle {}_a^{\rm C}\nabla _k^\alpha {x^{2m}}\left( k \right) \le 2m{x^{2m - 1}}\left( k \right){}_a^{\rm C}\nabla _k^\alpha x\left( k \right)},
\end{equation}
\begin{equation}\label{Eq51}
{\textstyle {}_a^{\rm C}\nabla _k^\alpha {x^{2}}\left( k \right) \le 2x\left( k \right){}_a^{\rm C}\nabla _k^\alpha x\left( k \right)},
\end{equation}
for any $0<\alpha<1$, $m\in\mathbb{Z}_+$, $x(k)\in\mathbb{R}$, $ k\in {\mathbb N}_{a} $ and $a\in\mathbb{R}$.
\end{corollary}
\begin{remark}\label{Remark 2}
(\ref{Eq14}) and (\ref{Eq49}) can be regarded as the discrete extension of that in \cite{Fernandez:2017CNSNS,Fernandez:2018CNSNS}. However, the presented results in \cite{Fernandez:2017CNSNS} are questionable and the conditions in \cite{Fernandez:2018CNSNS} are harsh and impractical. To overcome this drawback, two alternative formulas (\ref{Eq16}) and (\ref{Eq50}) are proposed. Likewise, (\ref{Eq16})-(\ref{Eq18}) are the generalization of the related results in \cite{Dai:2017NODY,Ding:2015CTA,Duarte:2015CNSNS}, respectively. The formula in (\ref{Eq51}) is the discrete case of (5) in \cite{Alikhanov:2012AMC} and (6) in \cite{Aguila:2014CNSNS}. Actually, the formula (\ref{Eq14}) can also be proven by means of (\ref{Eq4}) instead of (\ref{Eq3}), while it is difficult to derive the results in (\ref{Eq15}) and (\ref{Eq16}) like that.
\end{remark}

Note that the similar results on the definition of Riemann--Liouville can also be established.
\begin{theorem}\label{Theorem 2} The following inequalities hold
\begin{equation}\label{Eq52}
{\textstyle {}_a^{\rm R}\nabla _k^\alpha {x^{2m}}\left( k \right) \le 2x^m\left( k \right){}_a^{\rm R}\nabla _k^\alpha x^m\left( k \right)},
\end{equation}
\begin{equation}\label{Eq53}
{\textstyle {}_a^{\rm R}\nabla _k^\alpha {x^{\frac{{2m}}{n}}}\left( k \right) \le \frac{{2m}}{{2m - n}}x\left( k \right){}_a^{\rm R}\nabla _k^\alpha {x^{\frac{{2m}}{n} - 1}}\left( k \right)},
\end{equation}
\begin{equation}\label{Eq54}
{\textstyle {}_a^{\rm R}\nabla _k^\alpha {x^{\frac{{2m}}{n}}}\left( k \right) \le \frac{{2m}}{n}{x^{\frac{{2m}}{n} - 1}}\left( k \right){}_a^{\rm R}\nabla _k^\alpha x\left( k \right)},
\end{equation}
\begin{equation}\label{Eq55}
{\textstyle {}_a^{\rm R}\nabla _k^\alpha {x^{{2^m}}}\left( k \right) \le {2^m}{x^{{2^m} - 1}}\left( k\right){}_a^{\rm R}\nabla _k^\alpha x\left( k \right)},
\end{equation}
\begin{equation}\label{Eq56}
{\textstyle {}_a^{\rm R}\nabla _k^\alpha {y^{\rm{T}}}\left( k \right)Py\left( k \right) \le 2{y^{\rm{T}}}\left(k \right)P {}_a^{\rm R}\nabla _k^\alpha y\left( k \right)},
\end{equation}
for any $0<\alpha<1$, $m\in\mathbb{Z}_+$, $n\in\mathbb{Z}_+$, $2m\ge n$, $x(k)\in\mathbb{R}$, $y(k)\in\mathbb{R}^\kappa$, $\kappa\in\mathbb{Z}_+$, $ k\in {\mathbb N}_{a} $, $a\in\mathbb{R}$ and positive definite matrix $P\in\mathbb{R}^{\kappa\times \kappa}$.
\end{theorem}
\begin{proof} With the help of (\ref{Eq10}), the correctness of (\ref{Eq52})-(\ref{Eq56}) with $k=a$ can be checked directly. As a consequence, the case of $k\in\mathbb{N}_{a+1}$ is considered in the following discussion.

i) With the assistance of formulas (\ref{Eq6}) and (\ref{Eq12}), one has
\begin{equation}\label{Eq57}
{\textstyle \begin{array}{l}
{}_a^{\rm R}\nabla _k^\alpha {x^{2m}}\left( k \right) - 2x^m\left(k \right){}_a^{\rm R}\nabla _k^\alpha x^m\left( k \right)\\
={}_a^{\rm C}\nabla _k^\alpha {x^{2m}}\left( k \right)+\frac{{{{\left( {k - a + 1} \right)}^{\overline { - \alpha } }}}}{{\Gamma \left( { - \alpha  + 1} \right)}}{x^{2m}}\left( {a - 1} \right)\\
\hspace{12pt} - 2x^m\left(k \right){}_a^{\rm C}\nabla _k^\alpha x^m\left( k \right)-2{x^m}\left( k \right)\frac{{{{\left( {k - a + 1} \right)}^{\overline { - \alpha } }}}}{{\Gamma \left( { - \alpha  + 1} \right)}}{x^m}\left( {a - 1} \right) .
\end{array}}
\end{equation}
By applying the formula in (\ref{Eq26}), the previous formula becomes
\begin{equation}\label{Eq58}
{\textstyle \begin{array}{l}
{}_a^{\rm R}\nabla _k^\alpha {x^{2m}}\left( k \right) - 2x^m\left(k \right){}_a^{\rm R}\nabla _k^\alpha x^m\left( k \right)\\
= - \frac{{{{\left( {k - a + 1} \right)}^{\overline { - \alpha } }}}}{{\Gamma \left( { - \alpha  + 1} \right)}}[{x}\left( {k} \right)-{x}\left( {a - 1} \right)]^2 +\sum\nolimits_{j = a}^{k-1} { \frac{{{{\left( {k - j+1} \right)}^{\overline { - \alpha-1 } }}}}{{\Gamma \left( { - \alpha} \right)}}{g}\left( j \right)} \\
\hspace{12pt}+\frac{{{{\left( {k - a + 1} \right)}^{\overline { - \alpha } }}}}{{\Gamma \left( { - \alpha  + 1} \right)}}{x^{2m}}\left( {a - 1} \right)-2{x^m}\left( k \right)\frac{{{{\left( {k - a + 1} \right)}^{\overline { - \alpha } }}}}{{\Gamma \left( { - \alpha  + 1} \right)}}{x^m}\left( {a - 1} \right)\\
= - \frac{{{{\left( {k - a + 1} \right)}^{\overline { - \alpha } }}}}{{\Gamma \left( { - \alpha  + 1} \right)}}{x}^{2m}\left( k \right)+\sum\nolimits_{j = a}^{k-1} { \frac{{{{\left( {k - j+1} \right)}^{\overline { - \alpha-1 } }}}}{{\Gamma \left( { - \alpha} \right)}}{g}\left( j \right)} \\
\le 0.
\end{array}}
\end{equation}
This completes the proof of (\ref{Eq52}).

ii) In a similar way, it follows
\begin{equation}\label{Eq59}
{\textstyle
\begin{array}{l}
{}_a^{\rm R}\nabla _k^\alpha {x^{\frac{{2m}}{n}}}\left( k \right) - \frac{{2m}}{{2m - n}}x\left( k \right){}_a^{\rm R}\nabla _k^\alpha {x^{\frac{{2m}}{n} - 1}}\left( k \right)\\
={}_a^{\rm C}\nabla _k^\alpha {x^{\frac{{2m}}{n}}}\left( k \right) +\frac{{{{\left( {k - a + 1} \right)}^{\overline { - \alpha } }}}}{{\Gamma \left( { - \alpha  + 1} \right)}}{x^{\frac{{2m}}{n}}}\left( {a - 1} \right) \\
\hspace{12pt}- \frac{{2m}}{{2m - n}}x\left( k \right){}_a^{\rm C}\nabla _k^\alpha {x^{\frac{{2m}}{n} - 1}}\left( k \right)- \frac{{2m}}{{2m - n}}{x}\left( k \right)\frac{{{{\left( {k - a + 1} \right)}^{\overline { - \alpha } }}}}{{\Gamma \left( { - \alpha  + 1} \right)}}{x^{\frac{{2m}}{n} - 1}}\left( {a - 1} \right)\\
 = - \frac{{{{\left( {k - a + 1} \right)}^{\overline { - \alpha } }}}}{{\Gamma \left( { - \alpha  + 1} \right)}}{g}\left( {a - 1} \right) +\sum\nolimits_{j = a}^{k-1} { \frac{{{{\left( {k - j+1} \right)}^{\overline { - \alpha-1 } }}}}{{\Gamma \left( { - \alpha} \right)}}{g}\left( j \right)} \\
\hspace{12pt}+\frac{{{{\left( {k - a + 1} \right)}^{\overline { - \alpha } }}}}{{\Gamma \left( { - \alpha  + 1} \right)}}{x^{\frac{{2m}}{n}}}\left( {a - 1} \right)- \frac{{2m}}{{2m - n}}{x}\left( k \right)\frac{{{{\left( {k - a + 1} \right)}^{\overline { - \alpha } }}}}{{\Gamma \left( { - \alpha  + 1} \right)}}{x^{\frac{{2m}}{n} - 1}}\left( {a - 1} \right).
\end{array}}
\end{equation}
Recalling $g\left( k\right)$ in (\ref{Eq30}), formula (\ref{Eq59}) can be rewritten as
\begin{equation}\label{Eq60}
{\textstyle
\begin{array}{l}
{}_a^{\rm R}\nabla _k^\alpha {x^{\frac{{2m}}{n}}}\left( k \right) - \frac{{2m}}{{2m - n}}x\left( k \right){}_a^{\rm R}\nabla _k^\alpha {x^{\frac{{2m}}{n} - 1}}\left( k \right)\\
 = - \frac{n}{2m-n}\frac{{{{\left( {k - a + 1} \right)}^{\overline { - \alpha } }}}}{{\Gamma \left( { - \alpha  + 1} \right)}}x^{\frac{2m}{n}}\left( {k} \right) +\sum\nolimits_{j = a}^{k-1} { \frac{{{{\left( {k - j+1} \right)}^{\overline { - \alpha-1 } }}}}{{\Gamma \left( { - \alpha} \right)}}{g}\left( j \right)} \\
\le0.
\end{array}}
\end{equation}
This completes the proof of (\ref{Eq53}).

iii) Similarly, one has
\begin{equation}\label{Eq61}
{\textstyle
\begin{array}{l}
{}_a^{\rm R}\nabla _k^\alpha {x^{\frac{{2m}}{n}}}\left( k \right) - \frac{{2m}}{{ n}}x^{\frac{{2m}}{n} - 1}\left( k \right){}_a^{\rm R}\nabla _k^\alpha {x}\left( k \right)\\
={}_a^{\rm C}\nabla _k^\alpha {x^{\frac{{2m}}{n}}}\left( k \right) +\frac{{{{\left( {k - a + 1} \right)}^{\overline { - \alpha } }}}}{{\Gamma \left( { - \alpha  + 1} \right)}}{x^{\frac{{2m}}{n}}}\left( {a - 1} \right) \\
\hspace{12pt}- \frac{{2m}}{{n}}x^{\frac{{2m}}{n} - 1}\left( k \right){}_a^{\rm C}\nabla _k^\alpha {x}\left( k \right)- \frac{{2m}}{{n}}{x}^{\frac{{2m}}{n} - 1}\left( k \right)\frac{{{{\left( {k - a + 1} \right)}^{\overline { - \alpha } }}}}{{\Gamma \left( { - \alpha  + 1} \right)}}{x}\left( {a - 1} \right)\\
 = - \frac{{{{\left( {k - a + 1} \right)}^{\overline { - \alpha } }}}}{{\Gamma \left( { - \alpha  + 1} \right)}}{g}\left( {a - 1} \right) +\sum\nolimits_{j = a}^{k-1} { \frac{{{{\left( {k - j+1} \right)}^{\overline { - \alpha-1 } }}}}{{\Gamma \left( { - \alpha} \right)}}{g}\left( j \right)} \\
\hspace{12pt}+\frac{{{{\left( {k - a + 1} \right)}^{\overline { - \alpha } }}}}{{\Gamma \left( { - \alpha  + 1} \right)}}{x^{\frac{{2m}}{n}}}\left( {a - 1} \right)- \frac{{2m}}{{n}}{x}^{\frac{{2m}}{n} - 1}\left( k \right)\frac{{{{\left( {k - a + 1} \right)}^{\overline { - \alpha } }}}}{{\Gamma \left( { - \alpha  + 1} \right)}}{x}\left( {a - 1} \right)\\
 = - \frac{2m-n}{n}\frac{{{{\left( {k - a + 1} \right)}^{\overline { - \alpha } }}}}{{\Gamma \left( { - \alpha  + 1} \right)}}x^{\frac{2m}{n}}\left( {k} \right) +\sum\nolimits_{j = a}^{k-1} { \frac{{{{\left( {k - j+1} \right)}^{\overline { - \alpha-1 } }}}}{{\Gamma \left( { - \alpha} \right)}}{g}\left( j \right)} \\
\le0,
\end{array}}
\end{equation}
which completes the proof of (\ref{Eq54}).

iv) Setting $m=n=1$, then formulas (\ref{Eq52})-(\ref{Eq54}) arrive at
\begin{equation}\label{Eq62}
{\textstyle {}_a^{\rm R}\nabla _k^\alpha {x^{2}}\left( k \right) \le 2x\left( k \right){}_a^{\rm R}\nabla _k^\alpha x\left( k \right)}.
\end{equation}
On this basis, adopting the proof method in Theorem \ref{Theorem 1}, the remainder proof has been completed successfully. To avoid redundancy, it is omitted here.
\end{proof}
\begin{remark}\label{Remark 2}
The proof for Theorem \ref{Theorem 2} is a little more complicated than that of Theorem \ref{Theorem 1}, since the relationship between Riemann--Liouville fractional difference and Caputo fractional difference is adopted. Actually, the authors in \cite{Dai:2017NODY} tried to discuss the similar problem in the continuous case, while the easy-to-use result like \cite{Liu:2016NODY} was not given. Meanwhile, when $m=1$, formula (\ref{Eq52}) reduces to Lemma 2.10 of \cite{Wu:2017AMC}.
\end{remark}

Inspired by the discussion in \cite{Dai:2017NODY}, more practical results on the definition of Gr\"{u}nwald-Letnikov can be established.
\begin{theorem}\label{Theorem 3}
The following inequalities hold
\begin{equation}\label{Eq63}
{\textstyle {}_a^{\rm G}\nabla _k^\alpha {x^{2m}}\left( k \right) \le 2x^m\left( k \right){}_a^{\rm G}\nabla _k^\alpha x^m\left( k \right)},
\end{equation}
\begin{equation}\label{Eq64}
{\textstyle {}_a^{\rm G}\nabla _k^\alpha {x^{\frac{{2m}}{n}}}\left( k \right) \le \frac{{2m}}{{2m - n}}x\left( k \right){}_a^{\rm G}\nabla _k^\alpha {x^{\frac{{2m}}{n} - 1}}\left( k \right)},
\end{equation}
\begin{equation}\label{Eq65}
{\textstyle {}_a^{\rm G}\nabla _k^\alpha {x^{\frac{{2m}}{n}}}\left( k \right) \le \frac{{2m}}{n}{x^{\frac{{2m}}{n} - 1}}\left( k \right){}_a^{\rm G}\nabla _k^\alpha x\left( k \right)},
\end{equation}
\begin{equation}\label{Eq66}
{\textstyle {}_a^{\rm G}\nabla _k^\alpha {x^{{2^m}}}\left( k \right) \le {2^m}{x^{{2^m} - 1}}\left( k\right){}_a^{\rm G}\nabla _k^\alpha x\left( k \right)},
\end{equation}
\begin{equation}\label{Eq67}
{\textstyle {}_a^{\rm G}\nabla _k^\alpha {y^{\rm{T}}}\left( k \right)Py\left( k \right) \le 2{y^{\rm{T}}}\left(k \right)P {}_a^{\rm G}\nabla _k^\alpha y\left( k \right)},
\end{equation}
for any $0<\alpha<1$, $m\in\mathbb{Z}_+$, $n\in\mathbb{Z}_+$, $2m\ge n$, $x(k)\in\mathbb{R}$, $y(k)\in\mathbb{R}^\kappa$, $\kappa\in\mathbb{Z}_+$, $ k\in {\mathbb N}_{a} $, $a\in\mathbb{R}$ and positive definite matrix $P\in\mathbb{R}^{\kappa\times \kappa}$.
\end{theorem}
\begin{proof}
i) Setting $k\in\mathbb{N}_{a+1}$, Definition \ref{Definition 2} and formula (\ref{Eq6}) result in
\begin{equation}\label{Eq68}
{\textstyle \begin{array}{l}
{}_a^{\rm{G}}\nabla _k^\alpha {x^{2m}}\left( k \right) - 2{x^m}\left( k \right){}_a^{\rm{G}}\nabla _k^\alpha {x^m}\left( k \right)\\
 = \sum\nolimits_{j = 0}^{k - a} {\frac{{{{\left( {j + 1} \right)}^{\overline { - \alpha  - 1} }}}}{{\Gamma \left( { - \alpha } \right)}}{x^{2m}}\left( {k - j} \right)}  - 2{x^m}\left( k \right)\sum\nolimits_{j = 0}^{k - a} {\frac{{{{\left( {j + 1} \right)}^{\overline { - \alpha  - 1} }}}}{{\Gamma \left( { - \alpha } \right)}}{x^m}\left( {k - j} \right)} \\
 =  - {x^{2m}}\left( k \right) + \sum\nolimits_{j = 1}^{k - a} {\frac{{{{\left( {j + 1} \right)}^{\overline { - \alpha  - 1} }}}}{{\Gamma \left( { - \alpha } \right)}}\left[ {{x^{2m}}\left( {k - j} \right) - 2{x^m}\left( k \right){x^m}\left( {k - j} \right)} \right]}.
\end{array}}
\end{equation}

Because $\alpha\in(0,1)$, then the following equations hold
\begin{equation}\label{Eq69}
{\textstyle \frac{{{{\left( {j + 1} \right)}^{\overline { - \alpha  - 1} }}}}{{\Gamma \left( { - \alpha } \right)}} = \frac{{\Gamma \left( {j - \alpha } \right)}}{{\Gamma \left( { - \alpha } \right)\Gamma \left( {j + 1} \right)}} < 0},
\end{equation}
\begin{equation}\label{Eq70}
{\textstyle 2{x^m}\left( k \right){x^m}\left( {k - j} \right) \le {x^{2m}}\left( k \right) + {x^{2m}}\left( {k - j} \right)},
\end{equation}
for $j=1,2,\cdots,k-a$.

Substituting (\ref{Eq69}) and (\ref{Eq70}) into (\ref{Eq68}) yields
\begin{equation}\label{Eq71}
{\textstyle \begin{array}{l}
{}_a^{\rm{G}}\nabla _k^\alpha {x^{2m}}\left( k \right) - 2{x^m}\left( k \right){}_a^{\rm{G}}\nabla _k^\alpha {x^m}\left( k \right)\\
 \le  - {x^{2m}}\left( k \right) - \sum\nolimits_{j = 0}^{k - a} {\frac{{{{\left( {j + 1} \right)}^{\overline { - \alpha  - 1} }}}}{{\Gamma \left( { - \alpha } \right)}}{x^{2m}}\left( k \right)} \\
 =  - {x^{2m}}\left( k \right)\sum\nolimits_{j = 0}^{k - a} {\frac{{{{\left( {j + 1} \right)}^{\overline { - \alpha  - 1} }}}}{{\Gamma \left( { - \alpha } \right)}}} \\
 =  - {x^{2m}}\left( k \right)\sum\nolimits_{j = 0}^{k - a} {\nabla \frac{{{{\left( {j + 1} \right)}^{\overline { - \alpha } }}}}{{\Gamma \left( { - \alpha  + 1} \right)}}} \\
 =  - {x^{2m}}\left( k \right)\big[ {\frac{{{{\left( {k - a + 1} \right)}^{\overline { - \alpha } }}}}{{\Gamma \left( { - \alpha  + 1} \right)}} - \frac{{{0^{\overline { - \alpha } }}}}{{\Gamma \left( { - \alpha  + 1} \right)}}} \big]\\
 \le 0,
\end{array}}
\end{equation}
where $\frac{{{{\left( {k - a + 1} \right)}^{\overline { - \alpha } }}}}{{\Gamma \left( { - \alpha  + 1} \right)}}=1$ and $\frac{{{0^{\overline { - \alpha } }}}}{{\Gamma \left( { - \alpha  + 1} \right)}}=0$ are adopted here.

When $k=a$, (\ref{Eq63}) can be expressed as
\begin{equation}\label{Eq72}
{\textstyle {x^{2m}}\left( a \right) \le 2{x^m}\left( a \right){x^m}\left( a \right) = 2{x^{2m}}\left( a \right).}
\end{equation}
Till now, the proof of formula (\ref{Eq63}) has been completed.

ii) By applying Lemma \ref{Lemma 1}, one has
\begin{equation}\label{Eq73}
{\textstyle
x\left( k \right){x^{\frac{{2m}}{n} - 1}}\left( {k - j} \right) \le \frac{n}{{2m}}{x^{\frac{{2m}}{n}}}\left( k \right) + \frac{{2m - n}}{{2m}}{x^{\frac{{2m}}{n}}}\left( {k - j} \right).}
\end{equation}
With the help of (\ref{Eq69}) and (\ref{Eq73}), the formula in (\ref{Eq64}) can be updated as
\begin{equation}\label{Eq74}
{\textstyle
\begin{array}{l}
{}_a^{\rm{G}}\nabla _k^\alpha {x^{\frac{{2m}}{n}}}\left( k \right) - \frac{{2m}}{{2m - n}}x\left( k \right){}_a^{\rm{G}}\nabla _k^\alpha {x^{\frac{{2m}}{n} - 1}}\left( k \right)\\
 = \sum\nolimits_{j = 0}^{k - a} {\frac{{{{\left( {j + 1} \right)}^{\overline { - \alpha  - 1} }}}}{{\Gamma \left( { - \alpha } \right)}}{x^{\frac{{2m}}{n}}}\left( {k - j} \right)} \\
\hspace{12pt} - \frac{{2m}}{2m-n}{x}\left( k \right)\sum\nolimits_{j = 0}^{k - a} {\frac{{{{\left( {j + 1} \right)}^{\overline { - \alpha  - 1} }}}}{{\Gamma \left( { - \alpha } \right)}}x^{\frac{{2m}}{n} - 1}\left( {k - j} \right)} \\
\le  - \frac{n}{{2m - n}}{x^{\frac{{2m}}{n}}}\left( k \right) - \sum\nolimits_{j = 1}^{k - a} {\frac{{{{\left( {j + 1} \right)}^{\overline { - \alpha  - 1} }}}}{{\Gamma \left( { - \alpha } \right)}}\frac{n}{{2m - n}}{x^{\frac{{2m}}{n}}}\left( k \right)} \\
 =  - \frac{n}{{2m - n}}{x^{\frac{{2m}}{n}}}\left( k \right)\sum\nolimits_{j = 0}^{k - a} {\frac{{{{\left( {j + 1} \right)}^{\overline { - \alpha  - 1} }}}}{{\Gamma \left( { - \alpha } \right)}}} \\
 =  - \frac{n}{{2m - n}}{x^{\frac{{2m}}{n}}}\left( k \right)\frac{{{{\left( {k - a + 1} \right)}^{\overline { - \alpha } }}}}{{\Gamma \left( { - \alpha  + 1} \right)}}\\
 \le 0,
\end{array}}
\end{equation}
which confirms the result in (\ref{Eq64}) with $k\in\mathbb{N}_{a+1}$.

Given $k=a$, formula (\ref{Eq64}) becomes
\begin{equation}\label{Eq75}
{\textstyle
{x^{\frac{{2m}}{n}}}\left( a \right) \le \frac{{2m}}{{2m - n}}x\left( a \right){x^{\frac{{2m}}{n} - 1}}\left( a \right) = \frac{{2m}}{{2m - n}}{x^{\frac{{2m}}{n}}}\left( a \right),}
\end{equation}
where $\frac{{2m}}{{2m - n}}>1$. Consequently, (\ref{Eq64}) holds for any $k\in\mathbb{N}_{a}$.

iii) In a similar way, the proof of (\ref{Eq65}) can be preformed as follows
\begin{equation}\label{Eq76}
{\textstyle
\begin{array}{l}
{}_a^{\rm{G}}\nabla _k^\alpha {x^{\frac{{2m}}{n}}}\left( k \right) - \frac{{2m}}{{n}}{x^{\frac{{2m}}{n} - 1}}\left( k \right){}_a^{\rm{G}}\nabla _k^\alpha x\left( k \right)\\
 = \sum\nolimits_{j = 0}^{k - a} {\frac{{{{\left( {j + 1} \right)}^{\overline { - \alpha  - 1} }}}}{{\Gamma \left( { - \alpha } \right)}}{x^{\frac{{2m}}{n}}}\left( {k - j} \right)} \\
\hspace{12pt} - \frac{{2m}}{n}{x^{\frac{{2m}}{n} - 1}}\left( k \right)\sum\nolimits_{j = 0}^{k - a} {\frac{{{{\left( {j + 1} \right)}^{\overline { - \alpha  - 1} }}}}{{\Gamma \left( { - \alpha } \right)}}x\left( {k - j} \right)} \\
 = \sum\nolimits_{j = 1}^{k - a} {\frac{{{{\left( {j + 1} \right)}^{\overline { - \alpha  - 1} }}}}{{\Gamma \left( { - \alpha } \right)}}\big[ {{x^{\frac{{2m}}{n}}}\left( {k - j} \right) - \frac{{2m}}{n}{x^{\frac{{2m}}{n} - 1}}\left( k \right)x\left( {k - j} \right)} \big]} \\
\hspace{12pt} - \frac{{2m - n}}{n}{x^{\frac{{2m}}{n}}}\left( k \right)\\
 \le  - \frac{{2m - n}}{n}{x^{\frac{{2m}}{n}}}\left( k \right) - \sum\nolimits_{j = 1}^{k - a} {\frac{{{{\left( {j + 1} \right)}^{\overline { - \alpha  - 1} }}}}{{\Gamma \left( { - \alpha } \right)}}\frac{{2m - n}}{n}{x^{\frac{{2m}}{n}}}\left( k \right)} \\
 =  - \frac{{2m - n}}{n}{x^{\frac{{2m}}{n}}}\left( k \right)\sum\nolimits_{j = 0}^{k - a} {\frac{{{{\left( {j + 1} \right)}^{\overline { - \alpha  - 1} }}}}{{\Gamma \left( { - \alpha } \right)}}} \\
 =  - \frac{{2m - n}}{n}{x^{\frac{{2m}}{n}}}\left( k \right)\frac{{{{\left( {k - a + 1} \right)}^{\overline { - \alpha } }}}}{{\Gamma \left( { - \alpha  + 1} \right)}}\\
 \le 0.
\end{array}}
\end{equation}

When $k=a$, one obtains
\begin{equation}\label{Eq77}
{\textstyle {x^{\frac{{2m}}{n}}}\left( a \right) \le \frac{{2m}}{n}{x^{\frac{{2m}}{n} - 1}}\left( a \right)x\left( a \right) = \frac{{2m}}{n}{x^{\frac{{2m}}{n}}}\left( a \right)},
\end{equation}
where $\frac{{2m}}{n}\ge1$. All of these establish formula (\ref{Eq65}).

iv) Formulas (\ref{Eq63})-(\ref{Eq65}) can also converge as
\begin{equation}\label{Eq78}
{\textstyle {}_a^{\rm G}\nabla _k^\alpha {x^{2}}\left( k \right) \le 2x\left( k \right){}_a^{\rm G}\nabla _k^\alpha x\left( k \right)},
\end{equation}
when $m=n=1$. On this basis, adopting the proof method in Theorem \ref{Theorem 1}, the remainder proof on (\ref{Eq67}) and (\ref{Eq68}) has been completed successfully. For the sake of brevity, it is omitted here.
\end{proof}

\subsection{Natural extension}
With the Gr\"{u}nwald--Letnikov definition in (\ref{Eq5}), the needed initial conditions for ${}_a^{\rm C}\nabla _k^\alpha {x}\left( k \right)=f\left( x \right)$ is $x\left( a-1 \right)$ and its difference at $k=a-1$. To make full use of the initial condition at $k=a$, the Gr\"{u}nwald--Letnikov definition is modified by
\begin{equation}\label{Eq79}
{\textstyle {}_{a}^{\rm G}{\nabla}_{k}^{\alpha} f\left( k \right) \triangleq  \sum\nolimits_{j = 0}^{k-a-1} {{{\left( { - 1} \right)}^j}\big( {\begin{smallmatrix}
\alpha \\
j
\end{smallmatrix}} \big)f\left( {k - j} \right)}},
\end{equation}
where $\alpha\in\mathbb{R}$, $a\in\mathbb{R}$ and $ k\in {\mathbb N}_{a} $.

When $\alpha<0$, formula (\ref{Eq79}) is equal to the Riemann--Liouville sum in \cite{Wei:2019CNSNS}. On this basis, Riemann--Liouville difference and Caputo difference can be defined by using (\ref{Eq7}) and (\ref{Eq8}), respectively. Notably, for these newly built fractional differences with initial instant changed, Theorem \ref{Theorem 1} - Theorem \ref{Theorem 3} still hold.

Assume $ f:{\mathbb N}_{a-K} \to {\mathbb R}$, $\alpha\in\mathbb{R}$, $K\in {\mathbb Z}_{+} $, $k\in {\mathbb N}_{a} $ and $a\in\mathbb{R}$. Then, Gr\"{u}nwald--Letnikov fractional sum/differece can be defined by
\begin{equation}\label{Eq80}
{\textstyle {}_{k-K}^{\hspace{11pt}\rm G}{\nabla}_{k}^{\alpha} f\left( k \right) \triangleq  \sum\nolimits_{j = 0}^{K} {{{\left( { - 1} \right)}^j}\big( {\begin{smallmatrix}
\alpha \\
j
\end{smallmatrix}} \big)f\left( {k - j} \right)}},
\end{equation}

For $f : \mathbb{N}_{a-K-n} \to \mathbb{R}$, $n-1<\alpha<n$, $n\in\mathbb{Z}_+$, $K\in {\mathbb Z}_{+} $, $k\in {\mathbb N}_{a} $ and $a\in\mathbb{R}$, Riemann--Liouville difference and Caputo difference can be defined correspondingly
\begin{equation}\label{Eq81}
{}_{k-K}^{\hspace{11pt}\rm R}{\nabla}_k ^{\alpha} f\left( k \right) \triangleq {\nabla ^{n}}{}_{k-K}^{\hspace{11pt}\rm G}{\nabla}_k ^{ \alpha-n }f\left(k \right),
\end{equation}
\begin{equation}\label{Eq82}
{}_{k-K}^{\hspace{11pt}\rm C}{\nabla }_k ^{\alpha} f\left( k \right) \triangleq {}_{k-K}^{\hspace{11pt}\rm G}{\nabla}_k ^{ \alpha-n }{\nabla ^{n}}f\left( k \right).
\end{equation}
Notably, for these newly built fractional differences with fixed memory step in (\ref{Eq80})-(\ref{Eq82}), the similar results as Theorems \ref{Theorem 1}-\ref{Theorem 3} still hold.

Similarly, assume $ f:{\mathbb N}_{a} \to {\mathbb R}$, $\alpha(k)\in\mathbb{R}$, $k\in {\mathbb N}_{a} $ and $a\in\mathbb{R}$. Define Gr\"{u}nwald--Letnikov fractional sum/differece as
\begin{equation}\label{Eq83}
{\textstyle {}_{a}^{\rm G}{\nabla}_{k}^{\alpha(k)} f\left( k \right) \triangleq  \sum\nolimits_{j = 0}^{k-a} {{{\left( { - 1} \right)}^j}\big( {\begin{smallmatrix}
\alpha(k) \\
j
\end{smallmatrix}} \big)f\left( {k - j} \right)}},
\end{equation}
Furthermore, for $f : \mathbb{N}_{a-n(k)} \to \mathbb{R}$, $n(k)-1<\alpha(k) <n(k)$, $n(k)\in\mathbb{Z}_+$, $k\in {\mathbb N}_{a} $, $a\in\mathbb{R}$, Riemann--Liouville difference and Caputo difference can be defined
\begin{equation}\label{Eq84}
{}_a^{\rm R}{\nabla}_k ^{\alpha(k)} f\left( k \right) \triangleq {\nabla ^{n(k)}}{}_a^{\rm G}{\nabla}_k ^{ \alpha(k)-n(k) }f\left(k \right),
\end{equation}
\begin{equation}\label{Eq85}
{}_a^{\rm C}{\nabla }_k ^{\alpha(k)} f\left( k \right) \triangleq {}_a^{\rm G}{\nabla}_k ^{ \alpha(k)-n(k) }{\nabla ^{n(k)}}f\left( k \right),
\end{equation}
respectively. Notably, for these newly built fractional differences with time varying order in (\ref{Eq83}) and (\ref{Eq85}), the similar results as Theorem \ref{Theorem 1} and Theorem \ref{Theorem 3} still hold. Sadly, the corresponding result with Riemann--Liouville definition in  (\ref{Eq84}) is hard to develop. In pursuit of clear and concise, the related context is not provided here.

\begin{remark}\label{Remark 3}
Theorems \ref{Theorem 1}-\ref{Theorem 3} and Corollary \ref{Corollary 1} build the bridge from the fractional difference of the Lyapunov function to the system equation. With the newly built Lyapunov inequalities, the stability of the related system can be well solved. Interestingly, the three theorems exactly share the same form, which illustrates the uniformity of the considered definitions. Additionally, the elaborated results for nabla discrete fractional order systems might be beneficial for the continuous case.
\end{remark}
\section{Numerical Simulation} \label{Section 4}
In this section, three examples are given to show that it is convenient and efficient to validate the stability of nabla discrete fractional systems by using the proposed method. Note that neither the methods in \cite{Baleanu:2017CNSNS} nor \cite{Wu:2017AMC} are able to check the stability of the elaborated system. Accordingly, only the proposed methods are performed here.
\begin{example}
Consider the following system with $0<\alpha<1$
\begin{equation}\label{Eq86}
{\textstyle \left\{ \begin{array}{l}
{}_a^{\rm{C}}\nabla _k^\alpha {x_1}\left( k \right) =  - {x_1}\left( k \right) + x_2^3\left( k \right),\\
{}_a^{\rm{C}}\nabla _k^\alpha {x_2}\left( k \right) =  - {x_1}\left( k \right) - {x_2}\left( k \right),
\end{array} \right.}
\end{equation}
and select the Lyapunov candidate function as
\begin{equation}\label{Eq87}
{\textstyle V\left( k \right) = \frac{1}{2}x_1^2\left( k \right) + \frac{1}{4}x_2^4\left( k \right)}.
\end{equation}
Now, applying Theorem \ref{Theorem 1}, it can be found that
\begin{equation}\label{Eq88}
{\textstyle\begin{array}{rl}
{}_a^{\rm{C}}\nabla _k^\alpha V\left( k \right) \le&\hspace{-6pt} {x_1}\left( k \right){}_a^{\rm{C}}\nabla _k^\alpha {x_1}\left( k \right) + x_2^3\left( k \right){}_a^{\rm{C}}\nabla _k^\alpha {x_2}\left( k \right)\\
 = &\hspace{-6pt} - x_1^2\left( k \right) - x_2^4\left( k \right)\\
\le &\hspace{-6pt} 0.
\end{array}}
\end{equation}
As can be seen from (\ref{Eq88}), the fractional difference of the Lyapunov function is negative definite. By applying the Caputo case of Theorem 3.6 in \cite{Wu:2017AMC}, it can be concluded that the system (\ref{Eq86}) is asymptotically stable and the origin is the equilibrium point.

Fig. \ref{Fig 1} shows the evolution of the system state of (\ref{Eq86}) with $\alpha=0.8$, $a=0$, $ x_1(-1)=2$ and $x_2(-1)=-1$. As expected from the analytical analysis already presented, the system is asymptotically stable.

\begin{figure}[!htbp]
  \centering
  \includegraphics[width=0.8\textwidth]{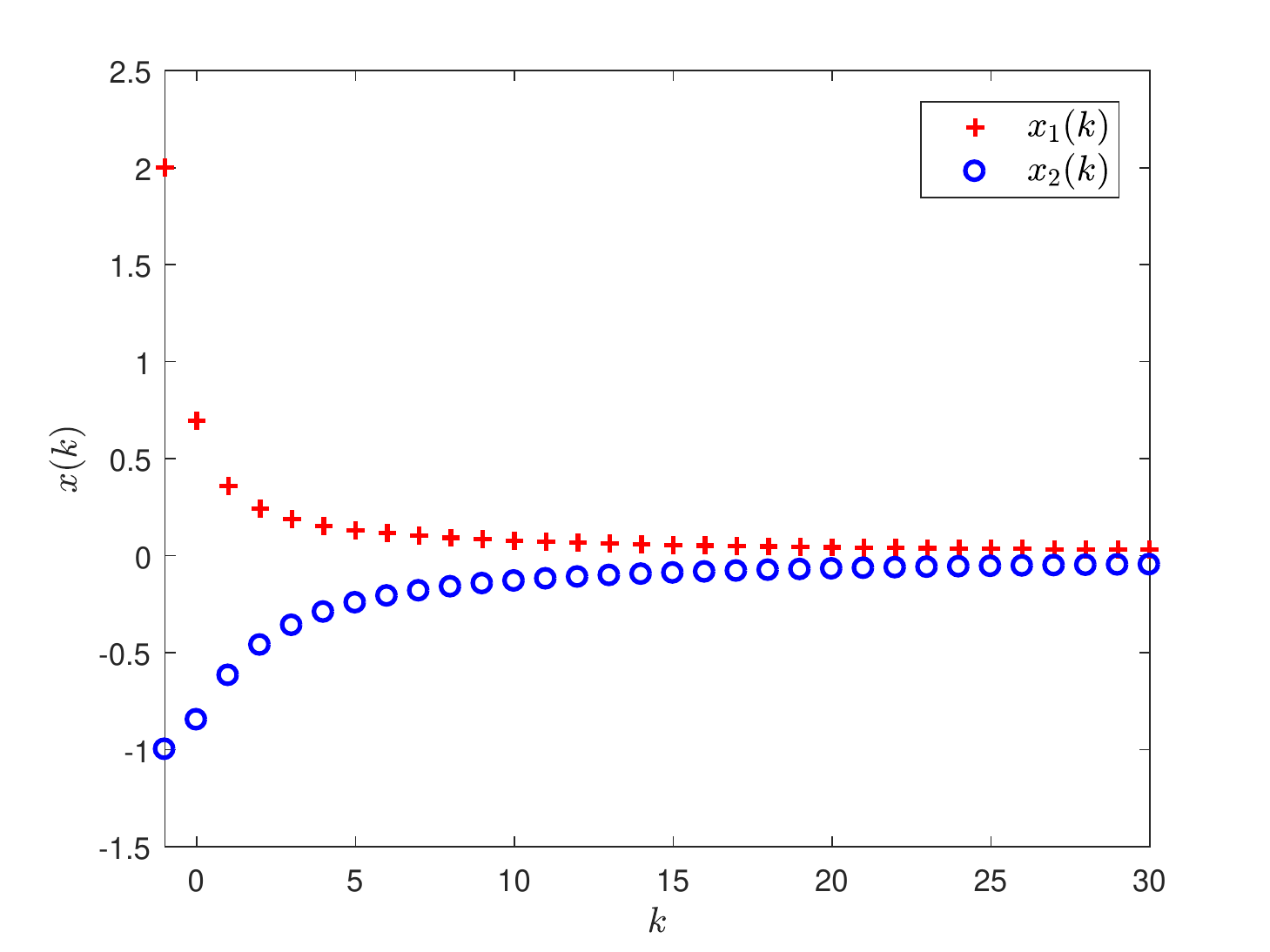}
  \caption{State response of the system in (\ref{Eq86}).}\label{Fig 1}
\end{figure}
\end{example}
\begin{example}
The previous example is borrowed from reference \cite{Aguila:2014CNSNS}. In order to make fully examine the obtained theoretical result, let us consider the following modified example
\begin{equation}\label{Eq89}
{\textstyle \left\{ \begin{array}{l}
{}_a^{\rm{C}}\nabla _k^\alpha {x_1}\left( k \right) =  - {x_1}\left( k \right) + x_2^{\frac{1}{3}}\left( k \right),\\
{}_a^{\rm{C}}\nabla _k^\alpha {x_2}\left( k \right) =  - x_1^{\frac{1}{5}}\left( k \right) - {x_2}\left( k \right),
\end{array} \right.}
\end{equation}
where $\alpha=0.8$, $a=0$, $ x_1(-1)=2$ and $x_2(-1)=-1$.

Using the following Laypunov function
\begin{equation}\label{Eq90}
{\textstyle V\left( k \right) = \frac{5}{6}x_1^{\frac{6}{5}}\left( k \right) + \frac{3}{4}x_2^{\frac{4}{3}}\left( k \right)},
\end{equation}
and following a similar procedure to analyze its fractional difference by applying (\ref{Eq16}), one obtains

\begin{equation}\label{Eq91}
{\textstyle \begin{array}{rl}
{}_a^{\rm{C}}\nabla _k^\alpha V\left( k \right) \le  &\hspace{-6pt} x_1^{\frac{1}{5}}\left( k \right){}_a^{\rm{C}}\nabla _k^\alpha {x_1}\left( k \right) + x_2^{\frac{1}{3}}\left( k \right){}_a^{\rm{C}}\nabla _k^\alpha {x_2}\left( k \right)\\
 = &\hspace{-6pt}  - x_1^{\frac{6}{5}}\left( k \right) - x_2^{\frac{4}{3}}\left( k \right)\\
\le  &\hspace{-6pt} 0.
\end{array}}
\end{equation}
It can be concluded that the system in (\ref{Eq89}) is stable in the Lyapunov sense. Fig. \ref{Fig 2} just illustrates this result.

\begin{figure}[!htbp]
  \centering
  \includegraphics[width=0.8\textwidth]{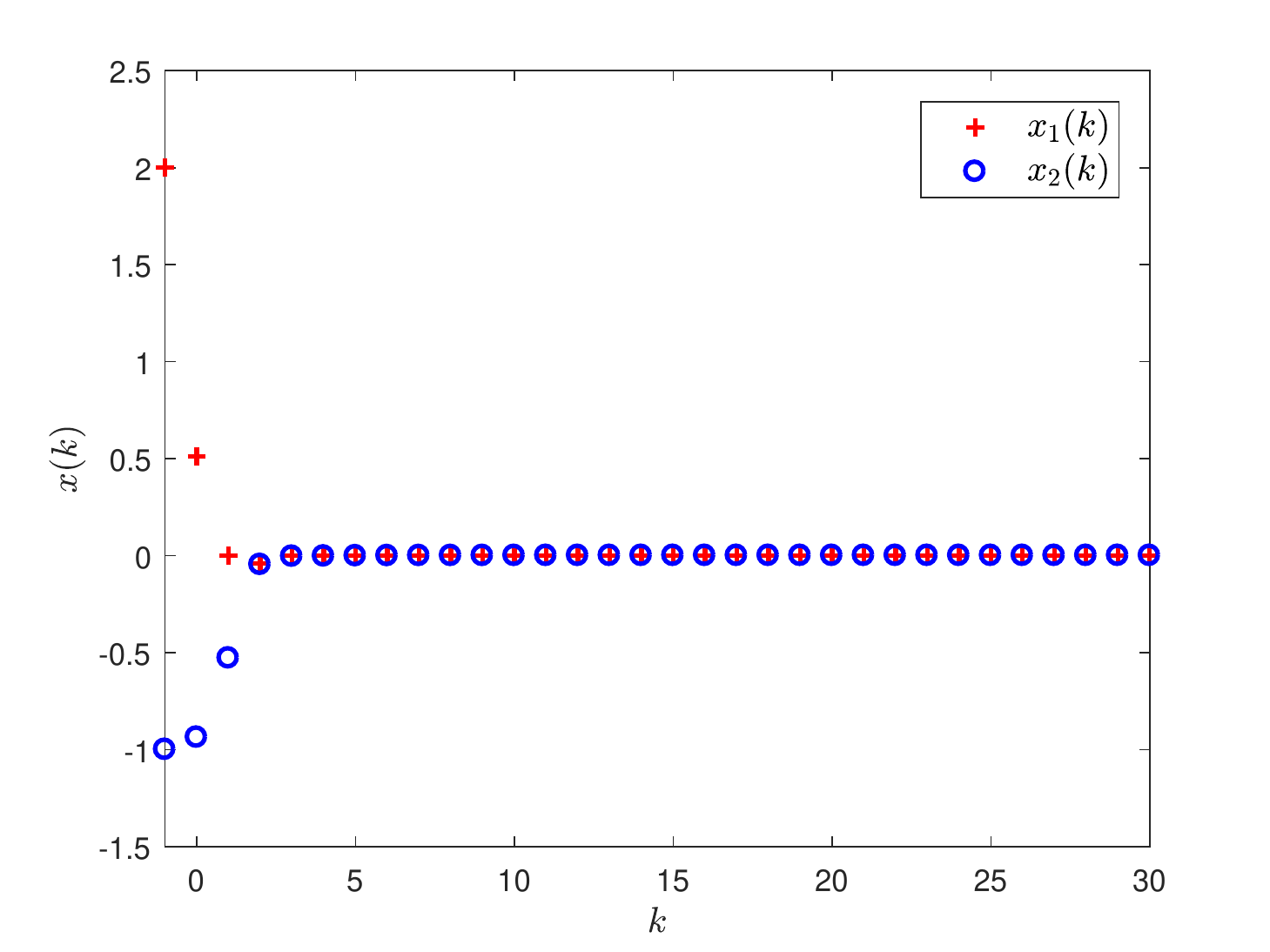}
  \caption{State response of the system in (\ref{Eq89}).}\label{Fig 2}
\end{figure}
\end{example}
\begin{example}
This example focuses on the multivariable optimization problem
\begin{equation}\label{Eq92}
{\textstyle { f\left( {x_1,x_2} \right) = {\left( {x_1 - 1} \right)^2} + 2{\left( {{x_1^2} - x_2} \right)^2}}}.
\end{equation}
Consider the fractional order gradient algorithm, i.e.,
\begin{equation}\label{Eq93}
{\textstyle \left\{ \begin{array}{l}
{}_a^{\rm{C}}\nabla _k^\alpha x_1\left( k \right) =  - \rho \frac{\partial }{{\partial x_1}}f\left( {x_1,x_2} \right),\\
{}_a^{\rm{C}}\nabla _k^\alpha x_2\left( k \right) =  - \rho \frac{\partial }{{\partial x_2}}f\left( {x_1,x_2} \right),
\end{array} \right.}
\end{equation}
where $\alpha=0.8$, $a=0$, $\rho=2$, $ x_1(-1)=2$ and $x_2(-1)=-1$.

Taking into account that the optimal value of $(x_1,x_2)$ is not zero, define $y_1\left( k \right) = x_1\left( k \right) - 1$ and $y_2\left( k \right) = x_2\left( k \right) - 1$.
Choosing the Laypunov function
\begin{equation}\label{Eq94}
{\textstyle V\left( k \right) = \frac{1}{4}{y_1^2}\left( k \right) + \frac{1}{2}{y_2^2}\left( k \right)},
\end{equation}
and using (\ref{Eq51}), it follows
\begin{equation}\label{Eq95}
{\textstyle \begin{array}{rl}
{}_a^{\rm C}\nabla _k^\alpha V\left( k \right) \le&\hspace{-6pt} \frac{1}{2}y_1\left( k \right){}_a^{\rm C}\nabla _k^\alpha y_1\left( k \right) + y_2\left( k \right){}_a^{\rm C}\nabla _k^\alpha y_2\left( k \right)\\
 =&\hspace{-6pt}  - \rho {\left[ {2{y_1^2}\left( k \right) + 3y_1\left( k \right) - 2y_2\left( k \right)} \right]^2}\\
 \le&\hspace{-6pt} 0.
\end{array}}
\end{equation}
Hence, one can conclude that the algorithm in (\ref{Eq93}) could find the desired point $(1,1)$ asymptotically. The corresponding curves in Fig. \ref{Fig 3} illuminate the results clearly.
\begin{figure}[!htbp]
  \centering
  \includegraphics[width=0.8\textwidth]{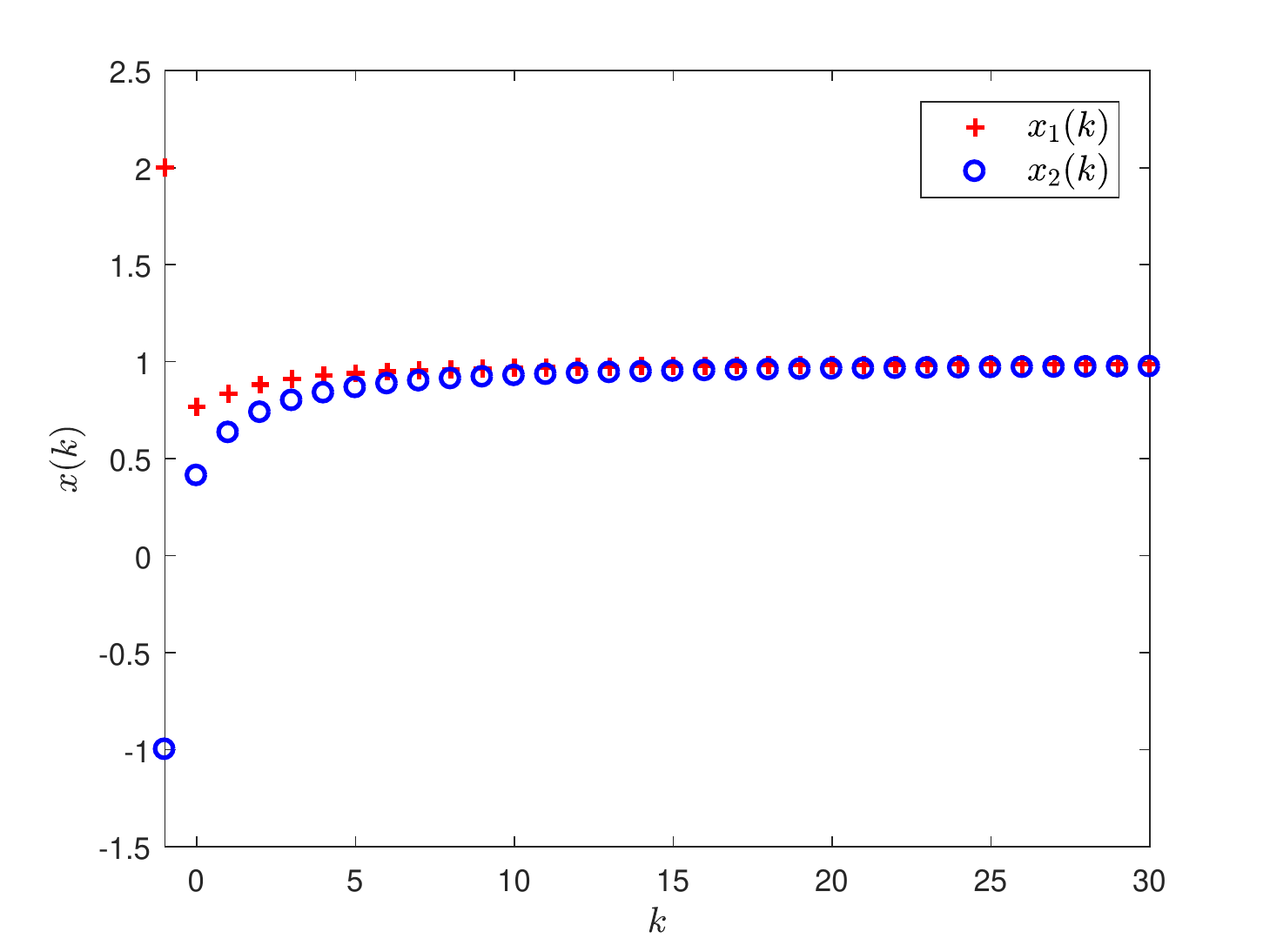}
  \caption{Search process of the algorithm in (\ref{Eq93}).}\label{Fig 3}
\end{figure}
\end{example}
\section{Conclusions}\label{Section 5}
In this paper, several useful inequalities on fractional difference of Lyapunov functions have been investigated. Note that, all the inequalities are applicable for Riemann--Liouville, Caputo and Gr\"{u}nwald--Letnikov definitions. Applying these results, the classical Lyapunov theory can be used to analyze the stability of discrete fractional nonlinear systems without and with delays. Three examples have shown that it is effective to check stability of diacrete fractional systems by using the proposed theory. It is believed that this work provides a convenient tool for analysis and synthesis for such discrete fractional calculus. Future research directions include discussing the inequalities on distributed order systems, incommensurate order systems, singular systems and time delay systems.

\section*{Acknowledgements}
The authors would like to thank the Associate Editor and the anonymous reviewers for their keen and insightful comments which greatly improved the contents and presentation. The work described in this paper was fully supported by the National Natural Science Foundation of China (No. 61601431, No. 61573332), the Anhui Provincial Natural Science Foundation (No. 1708085QF141), the Fundamental Research Funds for the Central Universities (No. WK2100100028) and the General Financial Grant from the China Postdoctoral Science Foundation (No. 2016M602032).

\section*{References}
\phantomsection
\addcontentsline{toc}{section}{References}
\bibliographystyle{model1-num-names}
\bibliography{database}

\end{document}